\documentclass[a4paper,10pt, reqno]{amsart}

\usepackage{hyperref}
\hypersetup{
    colorlinks=true,       
    linkcolor=blue,          
    citecolor=blue,        
    filecolor=blue,      
    urlcolor=blue           
}
\numberwithin{equation}{section}
\usepackage{rotating}
\usepackage[all]{xy}
\usepackage{tikz,graphicx}
\usepackage{amssymb}
\usepackage{booktabs}
\usepackage[english]{babel}
\usepackage{multirow}
\usepackage{multicol}
\usepackage{hyperref}
\usepackage{rotating}

\theoremstyle{plain}
\newtheorem{thm}{Theorem}[section]

\newtheorem{prop}[thm]{Proposition}
\newtheorem{cor}[thm]{Corollary}
\newtheorem{rem}[thm]{Remark}

\newtheorem{claim}{Claim}
\theoremstyle{definition}
\newtheorem{defn}[thm]{Definition}

\newcommand{\Z}{\mathbb{Z}}

\def\Aut{\operatorname{Aut}}
\def\PGL{\operatorname{PGL}}

\usepackage{enumerate}

\begin{document}
\title[Fake ES-irreducibile components of of smooth plane sextics]{On fake ES-irreducibile components of certain strata of smooth plane sextics}
\author[E. Badr] {Eslam Badr}
\address{$\bullet$\,\,E. Badr}
\address{Mathematics Department,
Faculty of Science, Cairo University, Giza-Egypt}
\email{eslam@sci.cu.edu.eg}
\address{Mathematics and Actuarial Science Department (MACT), American University in Cairo (AUC), New Cairo-Egypt}
\email{eslammath@aucegypt.edu}

\author[F. Bars] {Francesc Bars}
\address{$\bullet$\,\,F. Bars}
\address{Departament Matem\`atiques, Edif. C, Universitat Aut\`onoma de Barcelona\\
08193 Bellaterra, Catalonia}
\address{Barcelona Graduate School of Mathematics,
Catalonia} \email{francesc.bars@uab.cat}
%
%



\maketitle
\begin{abstract}
We construct the first examples of what we call \emph{fake ES-irreducible components}; Definition \ref{defnfake}. In our way to do so, we classify the automorphism groups of smooth plane sextics that only have automorphisms of order $\leq3$; Theorems \ref{thm1}, \ref{thm2} and \ref{thm3}, Corollaries \ref{cor312} and \ref{3123}.
\end{abstract}
\section{Introduction}
Let $\mathcal{M}_g^{\operatorname{Pl}}$ be the set of $K$-isomorphism classes of smooth plane curves $C$ of a fixed degree $d\geq4$. Here $K$ is an algebraically closed field of characterisitc $p=0$ or $p>2g+1$, where $g=(d-1)(d-2)/2\geq3$ is the geometric genus of $C$.

We can associate to any $[C]\in\mathcal{M}_g^{\operatorname{Pl}}$ infinitely many non-singular plane models, each of them is given by a homogeneous polynomial equation $C:F(X,Y,Z)=0$ of degree $d$ in $\mathbb{P}^2(K)$. Moreover, two such plane models for $C$ are $K$-isomorphic and their automorphism groups are $\operatorname{PGL}_3(K)$-conjugated via a projective change of variables $\phi\in\PGL_3(K)$.

Now, suppose that $G$ is a finite non-trivial group that can be embedded into $\PGL_3(K)$. We write $[C]\in\mathcal{M}_g^{\operatorname{Pl}}(G)$ when there exists an injective representation $\varrho:G\hookrightarrow\PGL_3(K)$ such that $\varrho(G)$ is a subgroup of $\Aut(C)$; the automorphism group of $C:F(X,Y,Z)=0$ inside $\PGL_3(K)$. More precisely, we say that $[C]$ belongs to the component $\mathcal{M}_g^{\operatorname{Pl}}(\varrho(G))$ of $\mathcal{M}_g^{\operatorname{Pl}}(G)$. Similarly, we write $[C]\in\widetilde{\mathcal{M}_g^{\operatorname{Pl}}}(G)$ when $\varrho(G)=\Aut(C)$ for some $\varrho$, and again we say that $[C]$ belongs to the component $\widetilde{\mathcal{M}_g^{\operatorname{Pl}}}(\varrho(G))$ of $\widetilde{\mathcal{M}_g^{\operatorname{Pl}}}(G)$.

Clearly, if $\varrho_{i}:G\hookrightarrow\operatorname{PGL}_3(K)$, for $i=1,2$, are $\operatorname{PGL}_3(K)$-conjugated, then $\mathcal{M}_g^{\operatorname{Pl}}(\varrho_1(G))=\mathcal{M}_g^{\operatorname{Pl}}(\varrho_2(G))$ and $\widetilde{\mathcal{M}_g^{\operatorname{Pl}}}(\varrho_1(G))=\widetilde{\mathcal{M}_g^{\operatorname{Pl}}}(\varrho_2(G))$. Accordingly,
$$
\mathcal{M}_g^{\operatorname{Pl}}(G)=\bigcup_{[\varrho]\in R_G}\,\mathcal{M}_g^{\operatorname{Pl}}\left(\varrho(G)\right)\,\,\text{and}\,\,\widetilde{\mathcal{M}_g^{\operatorname{Pl}}}(G)=\bigsqcup_{[\varrho]\in R_G}\,\widetilde{\mathcal{M}_g^{\operatorname{Pl}}}\left(\varrho(G)\right).
$$
Here $R_G:=\left\{\varrho:G\hookrightarrow\PGL_3(K)\right\}/\sim$, where $\varrho_1\sim\varrho_2$ if and only if $\varrho_1(G)$ and $\varrho_2(G)$ are $\PGL_3(K)$-conjugated.

\begin{defn}[ES-irreducibility \cite{MR3554328}]\label{ESdefn}
Each $[\varrho]\in R_G$ such that $\widetilde{\mathcal{M}_g^{\operatorname{Pl}}}\left(\varrho(G)\right)\neq\emptyset$ is called an \emph{ES-irreducible component} for $\widetilde{\mathcal{M}_g^{\operatorname{Pl}}}(G)$. We call $\widetilde{\mathcal{M}_g^{\operatorname{Pl}}}(G)$ \emph{ES-irreducible} if it has exactly one ES-irreducible component.
\end{defn}

Clearly, if a non-empty $\widetilde{\mathcal{M}_g^{\operatorname{Pl}}}(G)$ is not ES-irreducible, then it is not irreducible and the number of its ES-irreducible components is a lower bound for the number of its
irreducible components inside the coarse moduli space $\mathcal{M}_g$ of $K$-isomorphism classes of smooth curves of genus $g$.

Now, in the language of ES-irreducibility, one can interpret the results of Henn \cite{henn1976automorphismengruppen} and Komiya-Kuribayashi \cite{MR555703} for smooth plane quartic curves, which are genus $g=3$ curves, as follows: the strata $\widetilde{\mathcal{M}_3^{\operatorname{Pl}}}(G)$ are either empty or ES-irreducible. Thus each non-empty $\widetilde{\mathcal{M}_3^{\operatorname{Pl}}}(G)$ is described by a single \emph{normal form}; a homogenous polynomial equation $F(X,Y,Z)=0$ in $\mathbb{P}^2(K)$ equipped with parameters as its coefficients such that any $[C]\in\widetilde{\mathcal{M}_3^{\operatorname{Pl}}}(G)$ can be described by a smooth plane model through a specialization of those parameters.

\vspace{0.2cm}
\noindent\textbf{Notation.} Throughout the paper, $L_{i,B}$ denotes the generic homogeneous polynomial of degree $i$ in the variables $\{X,Y,Z\}-\{B\}$.

By $\zeta_n$ we mean a fixed primitive $n$th root of unity in $K$.

A projective linear transformation $A=(a_{i,j})\in\PGL_3(K)$ is sometimes written as
$$[a_{1,1}X+a_{1,2}Y+a_{1,3}Z:a_{2,1}X+a_{2,2}Y+a_{2,3}Z:a_{3,1}X+a_{3,2}Y+a_{3,3}Z].$$
For example, $[X:Z:Y]$ represents the projective change of variables $X\mapsto X,\,Y\mapsto Z,\,Z\mapsto Y$, and $\operatorname{diag}(1,a,b)$ represents $X\mapsto X,\,Y\mapsto aY,\,Z\mapsto bZ$ with $a,b\in K^*$.

We use the formal GAP library notations ``GAP$(n, m)$`` to refer the finite group of order $n$ that appears
in the $m$-th position of the atlas for small finite groups \cite{Gap}. See also \href{https://people.maths.bris.ac.uk/~matyd/GroupNames/index500.html}{GroupNames}.

Fix the following subgroups in $\PGL_3(K)$:

\begin{itemize}
\item $\varrho_1(\Z/2\Z):=\langle\operatorname{diag}(1,1,-1)\rangle$ and $\varrho_1((\Z/2\Z)^2):=\langle\varrho_1(\Z/2\Z),\,\operatorname{diag}(1,-1,1)\rangle$,
\item $\varrho_1(\Z/3\Z):=\langle\operatorname{diag}(1,1,\zeta_3)\rangle$ and $\varrho_1((\Z/3\Z)^2):=\langle\varrho_1(\Z/3\Z),\,\operatorname{diag}(1,\zeta_3,1)\rangle$,
\item $\varrho_2(\Z/3\Z):=\langle\operatorname{diag}(1,\zeta_3,\zeta_3^{-1})\rangle$ and $\varrho_{2}((\Z/3\Z)^2):=\langle\varrho_2(\Z/3\Z),[Y:Z:X]\rangle$,
\item $\varrho_1(\operatorname{S}_3):=\langle[Y:Z:X],\,[X:Z:Y]\rangle$ and $\varrho_2(\operatorname{S}_3):=\langle\varrho_2(\Z/3\Z),\,[X:Z:Y]\rangle$,
\item $\varrho_1(\Z/3\Z\rtimes\operatorname{S}_3):=\langle\varrho_1(\operatorname{S}_3),\,\varrho_2(\Z/3\Z)\rangle$,
\item $\varrho_{1}(\operatorname{A}_4):=\langle\varrho_1((\Z/2\Z)^2),\,[Y:Z:X]\rangle$ and $\varrho_{2}(\operatorname{A}_4):=\langle\varrho_1((\Z/2\Z)^2),\,[\zeta_6^{-1}Y:Z:X]\rangle$.
\end{itemize}
\vspace{0.2cm}

\begin{rem}
P. Henn observed that $\mathcal{M}_3^{\operatorname{Pl}}(\Z/3\Z)$ admits two ES-components. One component corresponds to $\varrho_1(\Z/3\Z)$ where any $[C]\in\mathcal{M}_3^{\operatorname{Pl}}(\varrho_1(\Z/3\Z))$ is given by an equation of the form $Z^3Y+L_{4,Z}=0$. The second component corresponds to $\varrho_2(\Z/3\Z)$ such that any $[C']\in\mathcal{M}_3^{\operatorname{Pl}}(\varrho_2(\Z/3\Z))$ is given by an equation of the form $X^4+ X(Y^3+Z^3)+\alpha_{2,1}X^2YZ+\alpha_{1,1}X(YZ)^2=0$ for some $\alpha_{2,1},\alpha_{1,1}\in K$. In particular, $C'$ has $[X:Z:Y]$ as an extra involution, thus $C'$ always has the symmetry group $\operatorname{S}_3$ as a subgroup of automorphisms. Therefore, $\widetilde{\mathcal{M}_3^{\operatorname{Pl}}}(\varrho_2(\Z/3\Z))=\emptyset$ and $\mathcal{M}_3^{\operatorname{Pl}}(\varrho_2(\Z/3\Z))\subseteq\mathcal{M}_3^{\operatorname{Pl}}(\operatorname{S}_3)$.
\end{rem}
Concerning smooth plane quintic curves, which are genus $g=6$ curves, Badr-Bars \cite{MR3508302} showed that all the strata $\widetilde{\mathcal{M}_6^{\operatorname{Pl}}}(G)$ are either empty or ES-irreducible except when $G=\Z/4\Z$. In this case, $\widetilde{\mathcal{M}_6^{\operatorname{Pl}}}(\Z/4\Z)$ has exactly two ES-irreducible components. Moreover, we generalized this result in \cite{MR3554328} for any odd degree $d\geq5$. More precisely, we proved that $\widetilde{\mathcal{M}_g^{\operatorname{Pl}}}(\Z/(d-1)\Z)$ has at least two ES-irreducible components for any $g=(d-1)(d-2)/2$ with $d\geq5$ odd. However, each of the strata $\widetilde{\mathcal{M}_6^{\operatorname{Pl}}}(\varrho(G))$ is described again by a single normal form.

Accordingly, we were wondering if this is the situation in general. That is to say, there always exists a single normal form describing the elements of $\widetilde{\mathcal{M}_g^{\operatorname{Pl}}}(\varrho(G))$ for each $\varrho\in R_G$. In this article, we will show that this impression is not true at least for smooth plane sextic curves, which are genus $g=10$ curves. We establish three counter examples corresponding to $G=\operatorname{\Z/3\Z}$ and $\operatorname{A}_4$ respectively.

On the other hand, classifying automorphism groups of smooth curves is a long standing problem that receives interest by many people. In the case of hyperelliptic curve, the structure of the automorphism group is quite explicit, see \cite{MR897252, MR1223022, MR2280308, MR2035219}. For non-hyperelliptic curves, we still have a lack of knowledge about the structure, except for some special cases. For example, the cases of low genus and also Hurwitz curves, see \cite{MR1796706, henn1976automorphismengruppen, MR839811, MR1072285, MR1068416}. This lack motivates us to do more investigation in this direction, especially for the case of smooth plane curves of degree $d\geq4$. In this paper, we classify the automorphism groups of smooth plane curves $C$ of degree $6$ such that $2$ and $3$ are the only divisors of $|\Aut(C)|$. A more detailed treatment of automorphisms of non-singular plane sextic curves is intended in \cite{BadrBarssextic}.

\vspace{0.5cm}
\noindent\textbf{Acknowledgments.} The authors are supported by PID2020-116542GB-I00, Ministerio de Ciencia y Universidades of Spanish
government. Also, E. Badr is partially supported by the School of Sciences and
Engineering (SSE) of the American Univeristy in Cairo (AUC).
\section{Main Results}
\begin{thm}\label{thm1}
Let $C$ be a smooth plane sextic curve that admits an automorphism of maximal order $2$. Up to $K$-isomorphism, $C$ is defined by an equation of the form:
$$
C:Z^6+Z^4L_{2,Z}+Z^2L_{4,Z}+L_{6,Z}=0
$$
such that $L_{6,Z}$ is of degree $\geq5$ in both $X$ and $Y$, and at least one of the binary forms $L_{2,Z}$ and $L_{4,Z}$ is non-zero. Moreover, $\Aut(C)=\varrho_1(\Z/2\Z)$ unless $L_{2,Z},\,L_{4,Z}$ and $L_{6,Z}$ belong to the ring $K[X^2,Y^2]$. In the latter case, $\Aut(C)=\varrho_1((\Z/2\Z)^2)$.
\end{thm}

\begin{cor}\label{cor201}
The strata $\widetilde{\mathcal{M}_{10}^{\operatorname{Pl}}}(\Z/2\Z)$ and $\widetilde{\mathcal{M}_{10}^{\operatorname{Pl}}}((\Z/2\Z)^2)$ are ES-irreducible.
\end{cor}
\begin{defn}[\cite{Mit}]
An homology of period $n$ is a projective linear transformation of the plane $\mathbb{P}^2(K)$, which is $\operatorname{PGL}_3(K)$-conjugate to
$\operatorname{diag}(1,1,\zeta_{n})$. Such a transformation fixes pointwise a line $\mathcal{L}$ (its axis) and a point $P$ off this line (its center). In its canonical form, $\mathcal{L}:Z=0$ and center $P=(0:0:1)$.
\end{defn}

Otherwise, it is called a \emph{non-homology}.

\begin{thm}\label{thm2}
Let $C$ be a smooth plane sextic curve that admits an homology of period $3$ as an automorphism of maximal order. Up to $K$-isomorphism, $C$ is defined by an equation of the form $Z^6+Z^3L_{3,Z}+L_{6,Z}=0$ where neither $L_{3,Z}$ nor $L_{6,Z}$ equals $0$. Moreover, $\Aut(C)$ is always $\varrho_1\left(\Z/3\Z\right)\rangle$ except when $C$ is $K$-isomorphic to $C'$ of the form
$
C':X^6+Y^6+Z^6+Z^3\left(\alpha_{3,0}X^3+\alpha_{0,3}Y^3\right)+\,\alpha_{3,3}X^3Y^3=0,
$
such that $\alpha_{3,0},\alpha_{0,3},\alpha_{3,3}$ are pair-wise distinct modulo $\{\pm1\}$. In this case, $\Aut(C')=\varrho_1\left((\Z/3\Z)^2\right)$.
\end{thm}
\begin{thm}\label{thm3}
Let $C$ be a smooth plane sextic curve that admits a non-homology of period $3$ as an automorphism of maximal order. Up to $K$-isomorphism, $C$ is a member of one of the following families:
  \begin{eqnarray*}
  \mathcal{C}_1&:&X^6+Y^6+Z^6+XYZ\left(\alpha_{4,1}X^3+\alpha_{1,4}Y^3+\alpha_{1,1}Z^3\right)+\alpha_{2,2}X^2Y^2Z^2\\
  &+&\alpha_{3,3}X^3Y^3+\alpha_{3,0}X^3Z^3+\alpha_{0,3}Y^3Z^3=0\\
  \mathcal{C}_2&:&X^5Y+Y^5Z+XZ^5+XYZ\left(\alpha_{3,2}X^2Y+\alpha_{1,3}Y^2Z+\alpha_{2,1}XZ^2\right)\\
  &+&\alpha_{2,4}X^2Y^4+\alpha_{0,2}Y^2Z^4+\alpha_{4,0}X^4Z^2=0.
  \end{eqnarray*}
In either way, $\sigma=\operatorname{diag}(1,\zeta_3,\zeta_3^{-1})$ is an automorphism of maximal order $3$.
\begin{enumerate}[(1)]
  \item The automorphism group $\Aut(\mathcal{C}_1)=\varrho_2(\Z/3\Z)$ except when one of the following conditions holds.
\begin{enumerate}[(i)]
\item If $\alpha_{4,1}=\alpha_{1,4}=\alpha_{1,1}=\alpha_{2,2}=0$ such that $\alpha_{3,3}\neq\alpha_{3,0}$, then $\mathcal{C}_1$ reduces to
          $$
       X^6+Y^6+Z^6+X^3\left(\alpha_{3,3}Y^3+\alpha_{3,0}Z^3\right)+\alpha_{0,3}Y^3Z^3=0,
      $$
where $\Aut(\mathcal{C}_1)=\varrho_1((\Z/3\Z)^2)$.
\item If \textbf{(a)} $\alpha_{4,1}=\pm\alpha_{1,4}\, \text{and}\,\alpha_{3,0}=\pm\alpha_{0,3},\,$ \textbf{(b)} $\alpha_{1,4}=\pm\alpha_{1,1}\,\, \text{and}\,\,\alpha_{3,3}=\pm\alpha_{3,0},$ or \textbf{(c)} $\alpha_{4,1}=\pm\alpha_{1,1}\,\, \text{and}\,\,\alpha_{3,3}=\pm\alpha_{0,3},$ then $\mathcal{C}_1$ is $K$-isomorphic to
    \begin{eqnarray*}
  \mathcal{C}'_1&:&X^6+Y^6+Z^6+\alpha'_{4,1}X^4YZ+\alpha'_{3,3}X^3(Y^3+Z^3)+\alpha'_{2,2}X^2Y^2Z^2\\
  &+& \alpha'_{1,2}XYZ(Y^3+Z^3)+\alpha'_{0,3}Y^3Z^3=0,
\end{eqnarray*}
where $\Aut(\mathcal{C}'_1)=\varrho_2(\operatorname{S}_3)$ if $\alpha'_{4,1}\neq\alpha'_{1,2}$ or $\alpha'_{3,3}\neq\alpha'_{0,3}$, and $\Aut(\mathcal{C}'_1)=\varrho_1(\Z/3\Z\rtimes\operatorname{S}_3)$ otherwise.
\begin{rem}
$\left(\alpha'_{3,3},\alpha'_{1,2}\right)\neq(0,0)$ or $\operatorname{diag}(1,\zeta_6,\zeta_6^{-1})$ will be an automorphism of order $6>3$.
\end{rem}
\item If  $\alpha_{4,1}=\zeta_6^{\ell}\alpha_{1,1},\,$$\alpha_{1,4}=\pm\zeta_6^{-\ell}\alpha_{1,1}$,\, $\alpha_{3,3}=\pm(-1)^{\ell}\alpha_{3,0},\,$ $\alpha_{0,3}=\pm\alpha_{3,0}$ for some $\ell\neq 0$ or $3\,\operatorname{mod}\,6$, then $\mathcal{C}_1$ is $K$-isomorphic to
\begin{eqnarray*}
\mathcal{C}''_1&:&X^6+\zeta_{6}^{2\ell}Y^6+\zeta_6^{-2\ell}Z^6+\alpha'_{1,1}XYZ(X^3+\zeta_6^{2\ell}Y^3+ \zeta_6^{-2\ell}Z^3)+\\
 &+&\alpha'_{3,0}(  X^3Y^3+\zeta_6^{-2\ell}X^3Z^3+ \zeta_6^{2\ell}Y^3Z^3)=0.
   \end{eqnarray*}
   where $\Aut(\mathcal{C}''_1)=\varrho_2((\Z/3\Z)^2)$.
    \item If \textbf{(a)} $(\alpha_{4,1},\alpha_{1,1},\alpha_{1,4}),\,$ $(\alpha_{1,4},\alpha_{4,1},\alpha_{1,1})$ or $(\alpha_{1,1},\alpha_{1,4},\alpha_{4,1})$ equals
    $$
    \left(\dfrac{2\left(29-54\lambda^6-54\mu^6\right)}{27\lambda\mu},\,\dfrac{2\left(27\mu^6-54\lambda^6-52\right)}{27\lambda\mu ^4},\,\dfrac{2\left(27\lambda^6-54\mu ^6-52\right)}{27\lambda^4\mu}\right),
    $$
\textbf{(b)} $(\alpha_{3,0},\alpha_{3,3},\alpha_{0,3}),\,$ $(\alpha_{3,3},\alpha_{0,3},\alpha_{3,0})$ or $(\alpha_{0,3},\alpha_{3,0},\alpha_{3,3})$ equals
$$
\left(\dfrac{2\left(81\lambda ^6-27\mu^6-26\right)}{27\mu ^3},\,\dfrac{2\left(81\mu^6-27\lambda^6-26\right)}{27\lambda^3},\,\dfrac{2\left(82-27\lambda^6-27\mu ^6\right)}{27\lambda^3\mu^3}\right),
$$
and \textbf{(c)} $\alpha_{2,2}=\dfrac{9\lambda^6+9\mu^6+10}{3\lambda^2\mu^2}$ for some $\lambda,\mu\in K^*$, then $\mathcal{C}_1$ is $K$-isomorphic to
\begin{eqnarray*}
  \mathcal{C}_{1,\lambda,\mu}: X^6+Y^6+Z^6&+&f_1(\lambda,\mu)X^2Y^2Z^2+f_2(\lambda,\mu)(X^4Y^2+X^2Z^4+Y^4Z^2)\\
&+&f_2(\mu,\lambda)(X^4Z^2+X^2Y^4+Y^2Z^4)=0,
  \end{eqnarray*}
where
\begin{eqnarray*}
f_1(\lambda,\mu)&:=&3(80+81\lambda^6+81\mu^6),\\
f_2(\lambda,\mu)&:=&81\left(1+\zeta_3\lambda^6+\zeta_3^{-1}\mu^6\right).
\end{eqnarray*}
In this case, $\Aut(\mathcal{C}_{1,\lambda,\mu})=\varrho_1(\operatorname{A}_{4})$.
\end{enumerate}

  \item The automorphism group $\Aut(\mathcal{C}_2)=\langle\sigma\rangle=\varrho_2(\Z/3\Z)$ except when one of the following conditions holds.
      \begin{enumerate}[(i)]
\item If $\alpha_{0,2}=\zeta_{21}^{-12r}\alpha_{4,0},\,$ $\alpha_{2,4}=\zeta_{21}^{3r}\alpha_{4,0},\,$ $\alpha_{1,3}=\zeta_{21}^{-6r}\alpha_{3,2},\,$ $\alpha_{2,1}=\zeta_{21}^{3r}\alpha_{3,2},$ then $\mathcal{C}_2$ is $K$-isomorphic to
    \begin{eqnarray*}
  \mathcal{C}'_2&:&X^5Y+Y^5Z+XZ^5+\alpha_{4,0}\zeta_{21}^{4r}\left(X^4Z^2+X^2Y^4+Y^2Z^4\right)\\
  &+& \alpha_{3,2}\zeta_{21}^{-r}XYZ\left(X^2Y+XZ^2+Y^2Z\right)=0,
\end{eqnarray*}
where $\Aut(\mathcal{C}'_2)=\varrho_{2}\left((\Z/3\Z)^2\right)$.
\begin{rem}
$\left(\alpha_{2,4},\alpha_{1,3}\right)\neq(0,0)$ or $\operatorname{diag}(1,\zeta_{21},\zeta_{21}^{-4})$ will be an automorphism of order $21>3$.
\end{rem}
\item If \textbf{(a)} $(\alpha_{2,4},\alpha_{4,0},\alpha_{0,2}),\,$ $(\alpha_{0,2},\alpha_{2,4},\alpha_{4,0})$ or $(\alpha_{4,0},\alpha_{0,2},\alpha_{2,4})$ equals
    $$
    \left(\dfrac{\lambda^5\mu+4\mu^5}{2\lambda^4},\,\dfrac{\lambda+4\lambda^5\mu}{2\mu^2},\,\dfrac{4\lambda+\mu^5}{2\lambda^2\mu^4}\right)
    $$
and \textbf{(b)} $(\alpha_{1,3},\alpha_{3,2},\alpha_{2,1}),\,$ $(\alpha_{2,1},\alpha_{1,3},\alpha_{3,2})$ or $(\alpha_{3,2},\alpha_{2,1},\alpha_{1,3})$ equals
$$
\left(\dfrac{2\left(2\lambda^5\mu+2\lambda+\mu^5\right)}{\lambda^3\mu^2},\,\dfrac{2\lambda^5\mu+4\lambda+4\mu^5}{\lambda^2\mu},\,\frac{2\left(2\lambda^5\mu+\lambda+2\mu^5\right)}{\lambda\mu^3}\right),
$$
then $\mathcal{C}_2$ is $K$-isomorphic to
\begin{eqnarray*}
  \mathcal{C}_{2,\lambda,\mu}: X^6+Y^6+Z^6&+&g_1(\lambda,\mu)(\zeta_3^{-1}X^4Y^2+X^2Z^4+Y^4Z^2)\\
  &+&g_2(\lambda,\mu)(X^4Z^2+\zeta_3X^2Y^4+Y^2Z^4)=0,
  \end{eqnarray*}
where
\begin{eqnarray*}
g_1(\lambda,\mu)&:=&\dfrac{\sqrt{3}\zeta_9\left(\zeta_4\lambda^5\mu+\zeta_{12}\lambda+\zeta_{12}^5\mu^5\right)}{\lambda^5\mu+\lambda+\mu^5},\\
g_2(\lambda,\mu)&:=&\dfrac{\sqrt{3}\zeta_{18}\left(\zeta_{12}^5\lambda^5\mu+\zeta_{12}\lambda+\zeta_4\mu^5\right)}{\lambda^5\mu+\lambda+\mu^5}.
\end{eqnarray*}
In this case, $\Aut(\mathcal{C}_{2,\lambda,\mu})=\varrho_2(\operatorname{A}_{4})$.
\end{enumerate}

\end{enumerate}

\end{thm}
We now introduce the notion of \emph{fake ES-irreducible components}.
\begin{defn}\label{defnfake}
An ES-irreducible component $\widetilde{\mathcal{M}_{g}^{\operatorname{Pl}}}(\varrho(G))$ is \emph{fake} if it is not defined by a single normal form.
\end{defn}
As a consequence of Theorems \ref{thm2} and \ref{thm3}:
\begin{cor}\label{cor312}
The strata $\widetilde{\mathcal{M}_{10}^{\operatorname{Pl}}}(\Z/3\Z)$ and $\widetilde{\mathcal{M}_{10}^{\operatorname{Pl}}}((\Z/3\Z)^2)$ are not ES-irreducible and each of them has exactly two ES-irreducible components namely, $\widetilde{\mathcal{M}_{10}^{\operatorname{Pl}}}(\varrho_i\left(\Z/3\Z\right))$ and $\widetilde{\mathcal{M}_{10}^{\operatorname{Pl}}}(\varrho_i\left((\Z/3\Z)^2\right))$ respectively with $i=1$ and $2$.

On the other hand, the components  $\widetilde{\mathcal{M}_{10}^{\operatorname{Pl}}}(\varrho_2\left(\Z/3\Z\right))$ and $\widetilde{\mathcal{M}_{10}^{\operatorname{Pl}}}(\varrho_2\left((\Z/3\Z)^2\right))$ are the first examples of fake ES-irreducible components. Any $[C]$ in the family $\mathcal{C}_2$ that belongs to $\widetilde{\mathcal{M}_{10}^{\operatorname{Pl}}}(\varrho_2\left(\Z/3\Z\right))$ or $\widetilde{\mathcal{M}_{10}^{\operatorname{Pl}}}(\varrho_2\left((\Z/3\Z)^2\right))$ has the property that its automorphism group  fixes the triangle $\triangle$ whose vertices $P_1=(1:0:0),\,P_2=(0:1:0)$ and $P_2=(0:0:1)$ lie on $C$. This does not hold if $[C]$ is in the family $\mathcal{C}_1$, in the sense that it is not necessarily true that $\Aut(C)=\varrho_2\left(\Z/3\Z\right)$ or $\varrho_2\left((\Z/3\Z)^2\right)$ fixes a triangle whose vertices lie on $C$. For example, take $[C]$ as in $\mathcal{C}''_1$ with $1+\alpha'_{1,1}+\alpha'_{3,0}\neq0$.
\end{cor}

\begin{cor}\label{3122}
The strata $\widetilde{\mathcal{M}_{10}^{\operatorname{Pl}}}(\operatorname{S}_3)$ and  $\widetilde{\mathcal{M}_{10}^{\operatorname{Pl}}}(\Z/3\Z\rtimes\operatorname{S}_3)$ are ES-irreducible. More precisely, $\widetilde{\mathcal{M}_{10}^{\operatorname{Pl}}}(\operatorname{S}_3)=\widetilde{\mathcal{M}_{10}^{\operatorname{Pl}}}(\varrho_2(\operatorname{S}_3))$ and
$\widetilde{\mathcal{M}_{10}^{\operatorname{Pl}}}(\Z/3\Z\rtimes\operatorname{S}_3)=\widetilde{\mathcal{M}_{10}^{\operatorname{Pl}}}(\varrho_1(\Z/3\Z\rtimes\operatorname{S}_3))$.
\end{cor}

\begin{cor}\label{3123}
The stratum $\widetilde{\mathcal{M}_{10}^{\operatorname{Pl}}}(\operatorname{A}_4)$ is ES-irreducible determined by $\widetilde{\mathcal{M}_{10}^{\operatorname{Pl}}}(\varrho_1(\operatorname{A}_4))$. It represents the second example of fake ES-irreducible components.
Indeed, $\mathcal{C}_{2,\lambda,\mu}$ is $K$-isomorphic, via a change of variables $\phi=\operatorname{diag}(1,s,t)$ such that $s=t^2$ and $t^3=\zeta_6$, to $^{\phi}\mathcal{C}_{2,\lambda,\mu}:X^6+\zeta_3^{-1}Y^6+\zeta_3Z^6+$ lower order terms, where $\Aut(^{\phi}\mathcal{C}_{2,\lambda,\mu})=\varrho_1(\operatorname{A}_4)$. Moreover, any $[C]\in\widetilde{\mathcal{M}_{10}^{\operatorname{Pl}}}(\varrho_1(\operatorname{A}_4))$ in the family $\mathcal{C}_{1,\lambda,\mu}$ is a descendant of the Fermat curve $\mathcal{F}_6$ in the sense of Theorem \ref{teoHarui} via a change of variables in the normalizer of $\varrho_1(\operatorname{A}_4)$ in $\PGL_3(K)$. This does not hold if $[C]$ is in the family $^{\phi}\mathcal{C}_{2,\lambda,\mu}$.
\end{cor}
\section{Preliminaries about automorphism groups}
Based entirely on geometrical methods, H. Mitchell \cite[\S 1-10]{Mit} proved that if $G$ is a finite subgroups of $\operatorname{PGL}_3(K)$, then it fixes a point, a line or a triangle unless it is primitive and conjugate to some group in a specific list. However, as a consequence of Maschke's theorem in group representation theory, the first two cases are equivalent, in the sense that if $G$ fixes a point (respectively a line), then it also
fixes a line not passing through the point (respectively a point not lying the line).

\vspace{0.2cm}
\noindent\textbf{Notation.} For a non-zero monomial $cX^{i_1}Y^{i_2}Z^{i_3}$ with $c\in K^*$, its exponent is defined to be $\operatorname{max}\{i_1,i_2,i_3\}$. For a homogenous polynomial $F(X,Y,Z)$, the core of it is
defined to be the sum of all terms of $F$ with the greatest
exponent. Now, let $C_0$ be a non-singular plane curve over $K$, a pair $(C,G)$ with
$G\leq \operatorname{Aut}(C)$ is said to be a descendant of $C_0$ if $C$ is defined
by a homogenous polynomial whose core is a defining polynomial of
$C_0$ and $G$ acts on $C_0$ under a suitable change of the
coordinates system, i.e. $G$ is $\operatorname{PGL}_3(K)$-conjugate to a subgroup of
$\operatorname{Aut}(C_0)$.

An element of $\operatorname{PGL}_3(K)$ is called \emph{intransitive} if it has the matrix shape
$$\left(
                                              \begin{array}{ccc}
                                                 1 & 0 & 0\\
                                                0 & \ast & \ast\\
                                                0 & \ast & \ast \\
                                              \end{array}
                                            \right).$$
The subgroup of $\operatorname{PGL}_3(K)$ of all intransitive elements is denoted by $\operatorname{PBD}(2,1)$. Obviously, there is a natural map $\Lambda:\operatorname{PBD}(2,1)\rightarrow \operatorname{PGL}_2(K)$ given by $$\left(
                                              \begin{array}{ccc}
                                                 1 & 0 & 0\\
                                                0 & \ast & \ast\\
                                                0 & \ast & \ast \\
                                              \end{array}
                                            \right)\in\operatorname{PBD}(2,1)\mapsto\left(
                                                                                                     \begin{array}{cc}
                                                                                                       \ast & \ast \\
                                                                                                       \ast& \ast \\
                                                                                                     \end{array}
                                                                                                   \right)\in \operatorname{PGL}_2(K).$$

Theorem \ref{teoHarui} below is very helpful for determining the full automorphism groups of smooth plane curves. For more details, we refer to the work of T. Harui \cite[Theorem 2.1]{Harui}.

\begin{thm}\label{teoHarui} Let $C$ be a non-singular plane curve of degree $d\geq4$ defined over an algebraically closed field $K$ of characteristic $0$. Then, one of the following situations holds:
\begin{enumerate}[1.]
  \item $\operatorname{Aut}(C)$ fixes a point on $C$ and then it is cyclic.
  \item $\operatorname{Aut}(C)$ fixes a point not lying on $C$ where we can think about $\Aut(C)$ in the following commutative diagram, with exact rows and vertical injective
morphisms:
  $$
\xymatrix
{
1\ar[r]  & K^*\ar[r]                    & \operatorname{PBD}(2,1)\ar[r]^{\Lambda}& \operatorname{PGL}_2(K)\ar[r]& 1         \\
         &                              &                                        &                              &\\
1\ar[r]  & N\ar[r]\ar@{^{(}->}[uu] & \operatorname{Aut}(C)\ar[r]\ar@{^{(}->}[uu] & G'\ar[r]\ar@{^{(}->}[uu]& 1
}
$$
Here, $N$ is a cyclic group of order dividing the degree $d$ and $G'$ is a
subgroup of $\operatorname{PGL}_2(K)$, which is conjugate to a cyclic group $\Z/m\Z$ of
order $m$ with $m\leq d-1$, a Dihedral group $\operatorname{D}_{2m}$ of order $2m$
with $|N|=1$ or $m|(d-2)$, one of the alternating groups $\operatorname{A}_4$, $\operatorname{A}_5$, or the symmetry group $\operatorname{S}_4$.
\begin{rem}\label{extensionrem}
We note that $N$ is viewed as the part of $\Aut(C)$ acting on the variable $B\in\{X,Y,Z\}$ and fixing the other two variables, while $G'$ is the part acting on $\{X,Y,Z\}-\{B\}$ and fixing $B$. For example, if $B=X$, then every automorphism in $N$ has the shape $\operatorname{diag}(\zeta_n,1,1)$ for some $n$th root of unity $\zeta_n$.
\end{rem}
\item $\operatorname{Aut}(C)$ is conjugate to a subgroup $G$ of $\operatorname{Aut}(\mathcal{F}_d)$, where $\mathcal{F}_d$ is the Fermat curve $X^d+Y^d+Z^d=0$. In particular, $|G|$ divides $|\operatorname{Aut}(\mathcal{F}_d)|=6d^2$, and $(C,G)$ is a descendant of $\mathcal{F}_d$.
\item $\operatorname{Aut}(C)$ is conjugate to a subgroup $G$ of $\operatorname{Aut}(\mathcal{K}_d)$, where $\mathcal{K}_d$ is the Klein curve curve
$X^{d-1}Y+Y^{d-1}Z+XZ^{d-1}=0$. In this case, $|\operatorname{Aut}(C)|$ divides $|\operatorname{Aut}(\mathcal{K}_d)|=3(d^2-3d+3)$, and
$(C,G)$ is a descendant of $\mathcal{K}_d$.
\item $\operatorname{Aut}(C)$ is conjugate to one of the finite primitive subgroup of $\operatorname{PGL}_3(K)$ namely, the Klein group
$\operatorname{PSL}(2,7)$, the icosahedral group $\operatorname{A}_5$, the alternating group
$\operatorname{A}_6$, or to one of the Hessian groups $\operatorname{Hess}_{*}$ with $*\in\{36,\,72,\,216\}$.
\end{enumerate}
\end{thm}
Finally, we have:
\begin{prop}\label{FermatandKlein}
The automorphism groups of the Fermat sextic curve $\mathcal{F}_6$
generated by $[X:Z:Y],\,[Y:Z:X],\,\operatorname{diag}(\zeta_6,1,1)$ and $\operatorname{diag}(1,\zeta_6,1)$ of orders $2,3,6$ and $6$ respectively is isomorphic to $\operatorname{GAP}(216,92)=(\Z/6\Z)^2\rtimes\operatorname{S}_3$. On the other hand, the automorphism group of the Klein sextic curve
$\mathcal{K}_6$ generated by $\operatorname{diag}(1,\zeta_{21},\zeta_{21}^{-4})$ and $[Y:Z:X]$ of orders $21$ and $3$ respectively is isomorphic to $\operatorname{GAP}(63,3)=\Z/21\Z\rtimes\Z/3\Z$.
\end{prop}
\begin{proof}
Regarding the generators of $\Aut(\mathcal{F}_6)$ and $\Aut(\mathcal{K}_6)$, we refer the reader to \cite[Propositions 3.3, 3.5]{Harui}.
Now, for the Fermat curve $\mathcal{F}_6$, take $a=[X:Z:Y],\,b=[Y:Z:X],\,c=\operatorname{diag}(\zeta_6,1,1)$ and $d=\operatorname{diag}(1,\zeta_6,1)$. One verifies that
$$(ab)^2=(ac)(ca)^{-1}=(cd)(dc)^{-1}=ada(cd)^{-5}=bcb^{-1}(cd)^{-5}=1.$$
These relations give us the $4$th semidirect product of $(\Z/6\Z)^2$ and $\operatorname{S}_3$ acting faithfully, see \href{https://people.maths.bris.ac.uk/~matyd/GroupNames/193/e15/S3byC6%5E2.html#s4}{semidirect products of $(\Z/6\Z)^2$ and $\operatorname{S}_3$}
for more details.

For the Klein curve $\mathcal{K}_6$, the two generators $a=\operatorname{diag}(1,\zeta_{21},\zeta_{21}^{-4})$ and $b=[Y:Z:X]$ of orders $21$ and $3$ respectively produce $\operatorname{GAP}(63,3)=\Z/21\Z\rtimes\Z/3\Z$ as $ba=(ab)^{-5}$.
\end{proof}
\section{Proof of Theorem \ref{thm2}}\label{case301}

In this case, $C:F(X,Y,Z)=0$ has an homology $\sigma$ of period $3$  in its automorphism group. The results in \cite{MR3475065} allows us to assume that $\sigma$ acts as $$(X:Y:Z)\mapsto(X:Y:\zeta_3Z)$$ up to $K$-isomorphism, where $\zeta_3$ is a fixed primitive $3$rd root of unity in $K$. In particular, $C$ is defined over $K$ by a non-singular plane equation of the form:
$$C:Z^6+Z^3L_{3,Z}+L_{6,Z}=0,$$
where $\sigma=\operatorname{diag}(1,1,\zeta_3)$ is an automorphism of maximal order $3$. By non-singularity, $L_{6,Z}$ should be of degree at least $5$ in both variables $X$ and $Y$. Also, $L_{3,Z}\neq0$ or $\operatorname{diag}(1,1,\zeta_6)$ would be an automorphism of order $6>3$.
%

In the sense of Theorem \ref{teoHarui}, we have the following:
\begin{itemize}
\item First, $\Aut(C)$ is not conjugate to any of the finite primitive subgroups of $\PGL_3(K)$ since each of them contains elements of order $>3$. Also, $C$ is not a descendant of the Klein sextic curve $\mathcal{K}_6$ because $\Aut(\mathcal{K}_6)$ by Proposition \ref{FermatandKlein} equals $\Z/21\Z\rtimes\Z/3\Z$ and it does not contains homologies of order $3$ similar to $\sigma$.
 \item Secondly, suppose that $C$ is a descendant of the Fermat curve $\mathcal{F}_6$. So there is a $\phi\in\PGL_3(K)$ such that $\phi^{-1}\Aut(C)\phi\leq\Aut(\mathcal{F}_6)$ and the transformed equation $^{\phi}C$ is $X^6+Y^6+Z^6+$ lower order terms in $X,Y,Z\,=0$. There is no loss of generality to impose $\phi^{-1}\langle\sigma\rangle\phi=\langle\sigma\rangle$ since homologies of period $3$ inside $\Aut(\mathcal{F}_6)$ form two conjugacy classes represented by $\sigma$ and $\sigma^{-1}$. Hence $^{\phi}C$ reduces to
     $$
     ^{\phi}C:X^6+Y^6+Z^6+Z^3L_{3,Z}+\,\text{lower order terms in}\,X, Y=0
     $$
Furthermore, by assumption, the automorphisms of $C$ have orders $\leq3$, then the group structure of $\Aut(\mathcal{F}_6)=(\Z/6\Z)^2\rtimes\operatorname{S}_3$ assures that $\Aut(^{\phi}C)$ would be one of the following groups inside $\Aut(\mathcal{F}_6)$: $$\Z/3\Z,\,(\Z/3\Z)^2,\,\operatorname{S}_3,\,\operatorname{A}_4,\,\Z/3\Z\rtimes\operatorname{S}_3,\,\operatorname{He}_3.$$ For more details, check the \href{https://people.maths.bris.ac.uk/~matyd/GroupNames/193/C6%5E2sS3.html}{subgroups lattice of $\Aut(\mathcal{F}_6)$}.

Now we tackle each of the above situations.

- Any copy of $\operatorname{S}_3$ (respectively $\operatorname{A}_4$) inside $\Aut(\mathcal{F}_6)$ is $\Aut(\mathcal{F}_6)$-conjugate to either $\varrho_i(\operatorname{S}_{3})$ (respectively $\varrho_i(\operatorname{A}_{4})$) with $i=1$ or $2$. But none of these subgroups has homologies of period $3$ similar to $\sigma$. So $\Aut(^{\phi}C)$ can not be an $\operatorname{S}_3$ or $\operatorname{A}_4$ inside $\Aut(\mathcal{F}_6)$.

- If $\Aut(^{\phi}C)$ equals a $(\Z/3\Z)^2,\,\Z/3\Z\rtimes\operatorname{S}_3$ or $\operatorname{He}_3$ in $\Aut(\mathcal{F}_6)$, then there must be $\sigma'\in\Aut(\mathcal{F}_6)\cap\Aut(^{\phi}C)$ of order $3$ that commutes with $\sigma$ as in any of these groups $\Z/3\Z$ is always contained in a $\left(\Z/3\Z\right)^2$. By Proposition \ref{FermatandKlein}, the elements of order $3$ in $\Aut(\mathcal{F}_6)$ are $\operatorname{diag}(1,s,t)$ with $s^3=t^3=1,$\,$[sY:tZ:X]$ and $[tZ:X:sY]$ with $s^6=t^6=1$. One easily verifies that only the diagonal shapes satisfies the description, equivalently, $\sigma'\in\langle\sigma,\operatorname{diag}(1,\zeta_3,1)\rangle$. In any case, we can reduce $C$ up to $K$-isomorphism  to
$$
     ^{\phi}C:X^6+Y^6+Z^6+Z^3\left(\alpha_{3,0}X^3+\alpha_{0,3}Y^3\right)+\,\alpha_{3,3}X^3Y^3=0,
     $$
where $\varrho_1\left((\Z/3\Z)^2\right)\leq\Aut(^{\phi}C)$.
\begin{rem}\label{important301}
In this scenario, the parameters $\alpha_{3,0},\alpha_{0,3},\alpha_{3,3}$ must be pair-wise distinct modulo $\{\pm1\}$ or $^{\phi}C$ will admit automorphisms of order $>3$. For example, $[\zeta_3Y:X:Z]\in\Aut(^{\phi}C)$ has order $6$ if $\alpha_{3,0}=\alpha_{0,3}$ and $[\zeta_3Y:X:-Z]\in\Aut(^{\phi}C)$ has order $6$ if $\alpha_{3,0}=-\alpha_{0,3}$.
\end{rem}
A similar discussion shows that any $\sigma''\in\Aut(\mathcal{F}_6)$ that commutes with $\sigma$ or $\sigma'$ belongs to $\langle\sigma,\sigma'\rangle$. Therefore, $\Aut(^{\phi}C)$ can not be the Heisenberg group $\operatorname{He}_3$ because this requires another automorphism $\sigma''\notin\langle\sigma,\sigma'\rangle$ that commutes with either $\sigma$ or $\sigma'$.

Finally, for $\Aut(^{\phi}C)$ to be $\Z/3\Z\rtimes\operatorname{S}_3$, it is necessary that $\Aut(\mathcal{F}_6)\cap\Aut(^{\phi}C)$ has involutions in it. Proposition \ref{FermatandKlein} tells us that the involutions of $\mathcal{F}_6$ are $\operatorname{diag}(-1,1,1),\,$ $\operatorname{diag}(1,-1,1),\,$ $\operatorname{diag}(1,1,-1),\,$ $[X:sZ:s^{-1}Y],\,$ $[s^{-1}Y:sX:Z]$ and $[sZ:Y:s^{-1}X]$ with $s^6=1$. If any of these involutions lies in $\Aut(^{\phi}C)$, then two of the parameters are equal modulo $\{\pm1\}$, which is absurd by Remark \ref{important301}. For example, $\operatorname{diag}(-1,1,1)\in\Aut(^{\phi}C)$ only if $\alpha_{3,0}=\alpha_{3,3}=0$,\,$[sY:s^{-1}X:Z]\in\Aut(^{\phi}C)$ only if $\alpha_{3,0}=\pm\alpha_{0,3}$, and so on.
\item Third, if $\Aut(C)$ fixes a line $\mathcal{L}$ and a point $P$ not lying on $\mathcal{L}$, then by Theorem \ref{teoHarui} we can think about $\Aut(C)$ in a short exact sequence $$1\rightarrow N=\langle\sigma\rangle\rightarrow \Aut(C)\rightarrow\Lambda(\Aut(C))\rightarrow 1,$$
where $\Lambda(\Aut(C))\simeq\Z/3\Z,\,\operatorname{D}_{4}$ or $\operatorname{A}_4$.

- Any group of order $36$ (respectively $12$) that has a normal subgroup isomorphic to $\Z/3\Z$ contains elements of order $6>3$, see \href{https://people.maths.bris.ac.uk/~matyd/GroupNames/1/Dic3.html}{Groups of order $12$} and \href{https://people.maths.bris.ac.uk/~matyd/GroupNames/1/C36.html}{Groups of order $36$} for more details. This allows us to exclude that $\Lambda(\Aut(C))$ equals $\operatorname{A}_4$ or $\operatorname{D}_4$.

- On the other hand, if $\Lambda(\Aut)(C)$ equals $\Z/3\Z$ in $\PGL_2(K)$, then $\Aut(C)$ equals $(\Z/3\Z)^2$ in $\operatorname{PBD}(2,1)$. In particular, $C:Z^6+Z^3L_{3,Z}+L_{6,Z}=0$ admits an automorphism $\sigma'\in\operatorname{PBD}(2,1)-\langle\sigma\rangle$ of order $3$ that commutes with $\sigma$. Depending on whether $\sigma'$ is an homology or a non-homology, it is conjugate via a change of variables $\phi\in\operatorname{PBD}(2,1)$, the normalizer of $\langle\sigma\rangle$, to $\operatorname{diag}(1,\zeta_3,1)$ or $\operatorname{diag}(1,\zeta_3,\zeta_3^{-1})$ respectively. In either way, $\Aut(^{\phi}C)=\varrho_1\left((\Z/3\Z)^2\right)$ which appeared earlier.
\end{itemize}

Summing up, we deduce that $\Aut(C)$ is always cyclic of order $3$ generated by $\sigma$ except when $C$ is projectively equivalent to $C'$ of the form
$$
C':X^6+Y^6+Z^6+Z^3\left(\alpha_{3,0}X^3+\alpha_{0,3}Y^3\right)+\,\alpha_{3,3}X^3Y^3=0,
$$
such that $\alpha_{3,0},\alpha_{0,3},\alpha_{3,3}$ are pair-wise distinct modulo $\{\pm1\}$. In this case, $\Aut(C)$ is conjugate to $(\Z/3\Z)^2$ generated by $\operatorname{diag}(1,\zeta_3,1)$ and $\operatorname{diag}(1,\zeta_3,1)$.

This proves Theorem \ref{thm2}.

\section{Proof of Theorem \ref{thm1}}\label{case201}
In this case, $C:F(X,Y,Z)=0$ has an homology $\sigma$ of period $2$  in its automorphism group. By \cite{MR3475065}, there is no loss of generality to assume that $\sigma$ acts as $$(X:Y:Z)\mapsto(X:Y:-Z)$$ up to $K$-isomorphism. In particular, $C$ is defined over $K$ by a non-singular plane equation of the form:
$$C:Z^6+Z^4L_{2,Z}+Z^2L_{4,Z}+L_{6,Z}=0$$
where $\sigma=\operatorname{diag}(1,1,-1)$ is an automorphism of maximal order $2$. Again $L_{6,Z}$ is of degree $\geq5$ in $X$ and $Y$ by non-singularity. Also, $L_{2,Z}$ or $L_{4,Z}$ does not vanish or $\operatorname{diag}(1,1,\zeta_6)$ will be an automorphism of order $6>3$ otherwise.
\begin{itemize}
  \item Obviously, $\Aut(C)$ is not conjugate to any of the finite primitive subgroups of $\PGL_3(K)$ as each of them contains elements of order $>2$. Also, $C$ can not be a descendant of the Klein sextic curve $\mathcal{K}_6$ since $2\nmid|\Aut(\mathcal{K}_6)|$, recall that $|\Aut(\mathcal{K}_6)|=63$ by Proposition \ref{FermatandKlein}.
  \item Secondly, if $\Aut(C)$ fixes a line $\mathcal{L}$ and a point $P$ off $\mathcal{L}$, then, by Theorem \ref{teoHarui}, $\Aut(C)$ is inside $\operatorname{PBD}(2,1)$ and satisfies a short exact sequence $$1\rightarrow N=\langle\sigma\rangle\rightarrow \Aut(C)\rightarrow\Lambda(\Aut(C))\rightarrow 1.$$
Our assumptions that any automorphism of $C$ has order $\leq2$ implies that $\Lambda(\Aut(C))$ is either $\Z/2\Z$ or $\operatorname{D}_{4}$ inside $\PGL_2(K)$, so $\Aut(C)$ is conjugate to either $(\Z/2\Z)^2$ or $(\Z/2\Z)^3$. In both situations $\Aut(C)$ has another involution $\sigma'$ that commutes with $\sigma$. Up to projective equivalence via a change of variables $\phi\in\operatorname{PBD}(2,1)$, the normalizer of $\langle\sigma\rangle$ in $\PGL_3(K)$, we can assume that $\sigma'=\operatorname{diag}(1,-1,1)$. Consequently, $C$ is $K$-isomorphic to $C':Z^6+Z^4L_{2,Z}+Z^2L_{4,Z}+L_{6,Z}=0$ for some $L_{i,Z}\in K[X^2,Y^2]$. Moreover, $\Aut(C)$ equals $(\Z/2\Z)^3$ only if there is an involution $\sigma''\notin\operatorname{PBD}(2,1)-\langle\sigma,\sigma'\rangle$ that commutes with both $\sigma$ and $\sigma'$. It is straightforward to check that such $\sigma''$ does not exist, hence $\Aut(C)$ is not $(\Z/2\Z)^3$ in this case.

\item If $C$ is a descendant of the Fermat curve $\mathcal{F}_6$ via a change of variables $\phi\in\PGL_3(K)$ with bigger automorphism group than $\langle\sigma\rangle$, then $\Aut(^{\phi}C)$ is a copy of $(\Z/2\Z)^2$ inside $\Aut(\mathcal{F}_6)$. Indeed any other subgroup of $\Aut(\mathcal{F}_6)$ has elements of order $>2$, see \href{https://people.maths.bris.ac.uk/~matyd/GroupNames/193/C6%5E2sS3.html}{subgroups lattice of $\Aut(\mathcal{F}_6)$}.

Up to $\Aut(\mathcal{F}_6)$-conjugation, there are two copies of $(\Z/2\Z)^2$ inside $\Aut(\mathcal{F}_6)$ namely, $\langle\sigma,\sigma'\rangle$ and $\langle\sigma,\tau\rangle$ with $\sigma'=\operatorname{diag}(1,-1,1)$ and $\tau=[Y:X:Z]$. However, both groups are $\PGL_3(K)$-conjugated via a transformation in $\operatorname{PBD}(2,1)$, the normalizer of $\langle\sigma\rangle$ in $\PGL_3(K)$. Thus there is no loss of generality to assume that $\Aut(C)$ is conjugate to $\varrho_1\left((\Z/2\Z)^2\right)$, which was treated earlier.
\end{itemize}

Summing up, we deduce that $\Aut(C)$ is always cyclic of order $2$ generated by $\sigma$ except when $L_{i,Z}\in K[X^2,Y^2]$ for $i=2,4,6$. In the latter case, $\Aut(C)$ equals $\varrho_1\left((\Z/2\Z)^2\right)$, which shows Theorem \ref{thm1}.
\section{Proof of Theorem \ref{thm3}}\label{case312}
In this case, $C:F(X,Y,Z)=0$ has a non-homology $\sigma$ of period $3$ in its automorphism group. By \cite{MR3475065}, one can assume that $\sigma$ acts as $$(X:Y:Z)\mapsto(X:\zeta_3Y:\zeta_3^{-1}Z)$$ up to $K$-isomorphism, where $\zeta_3$ is a fixed primitive $3$rd root of unity in $K$. In particular, $C$ is a $K$-isomorphic to a non-singular plane model in one of the following families:
  \begin{eqnarray*}
  \mathcal{C}_1&:&X^6+Y^6+Z^6+XYZ\left(\alpha_{4,1}X^3+\alpha_{1,4}Y^3+\alpha_{1,1}Z^3\right)+\alpha_{2,2}X^2Y^2Z^2\\
  &+&\alpha_{3,3}X^3Y^3+\alpha_{3,0}X^3Z^3+\alpha_{0,3}Y^3Z^3=0\\
  \mathcal{C}_2&:&X^5Y+Y^5Z+XZ^5+XYZ\left(\alpha_{3,2}X^2Y+\alpha_{1,3}Y^2Z+\alpha_{2,1}XZ^2\right)\\
  &+&\alpha_{2,4}X^2Y^4+\alpha_{0,2}Y^2Z^4+\alpha_{4,0}X^4Z^2=0.
  \end{eqnarray*}
where $\sigma:=\operatorname{diag}(1,\zeta_3,\zeta_3^{-1})$ is an automorphism of maximal order $3$.

\begin{itemize}
\item Again $\Aut(\mathcal{C}_i)$ for $i=1$ and $2$ is not conjugate to any of the finite primitive subgroups of $\PGL_3(K)$.
\item Suppose that $\Aut(\mathcal{C}_i)$ fixes a line $\mathcal{L}$ and a point $P$ not lying on this line. Since $\sigma$ is a non-homology inside $\Aut(\mathcal{C}_i)$ in its canonical form, $\mathcal{L}$ must be one of the reference lines; $B=0$ with $B=X,Y$ or $Z$ and $P$ is the reference point $(1:0:0),\,(0:1:0)$ or $(0:0:1)$ respectively.

   - For $\mathcal{C}_2$, the point $P$ belongs to $C:F(X,Y,Z)=0$. Hence $\Aut(\mathcal{C}_2)$ is cyclic, generated by $\langle\sigma\rangle$.

   - For $\mathcal{C}_1$, we can further impose $\mathcal{L}:X=0$ and $P=(1:0:0)$ (in the worst case scenario, one just needs to permute two of the variables and to fix the third one, which preserves the property that $\sigma$ remains an automorphism). In particular, by Theorem \ref{teoHarui}, $\Aut(\mathcal{C}_1)\subseteq\operatorname{PBD}(2,1)$ and lives in a short exact sequence:
   $1\rightarrow N\rightarrow \Aut(\mathcal{C}_1)\rightarrow\Lambda(\Aut(\mathcal{C}_1))\rightarrow 1,$
   where $N=\langle\tau\rangle$ has order $1,2$ or $3$ and $\Lambda(\Aut(C))$ is either $\Z/3\Z$,\, $\operatorname{S}_{3}$ with $|N|=1$ or $\operatorname{A}_4$ in $\PGL_2(K)$.
First, we easily exclude the case when $\tau$ has order $2$ because $\sigma\tau$ would be an automorphism of order $6>3$, a contradiction. 

Secondly, we handle each of the remaining cases:
\begin{enumerate}
  \item If $\Lambda(\Aut(\mathcal{C}_1))=\Z/3\Z$ and $N=1$, then $\Aut(\mathcal{C}_1)=\Z/3\Z$ generated by $\sigma$.
  \item If $\Lambda(\Aut(\mathcal{C}_1))=\Z/3\Z$ and $N=\Z/3\Z$, then $\Aut(\mathcal{C}_1)=\varrho_1((\Z/3\Z)^2)$ generated by $\sigma$ and $\tau=\operatorname{diag}(\zeta_3,1,1)$. In particular, $\alpha_{4,1}=\alpha_{2,2}=\alpha_{1,1}=\alpha_{1,4}=0$, and $\mathcal{C}_1$ reduces to
      $$
X^6+Y^6+Z^6+Z^3\left(\alpha_{3,0}X^3+\alpha_{0,3}Y^3\right)+\,\alpha_{3,3}X^3Y^3=0,
$$
which happened before in Theorem \ref{thm2}. We also remark that $\alpha_{3,0}\neq\alpha_{0,3}$ or $[Y:X:Z]$ will be an extra involution for $\mathcal{C}_1$.

This clarifies part of Theorem \ref{thm3}, (1)-(i).
  \item If $\Lambda(\Aut(\mathcal{C}_1))=\operatorname{S}_{3}$ and $N=1$, then $C$ should have an involution $\tau$ such that $\tau\sigma\tau=\sigma^{-1}$. So $\tau=[X:sZ:s^{-1}Y],\,[sY:s^{-1}X:Z]$ or $[sZ:Y:s^{-1}X]$ with $s^6=1$. This holds if we are in one of the situations:
      $\alpha_{3,3}=\pm\alpha_{3,0}$ and $\alpha_{1,1}=\pm\alpha_{1,4}$,\,$\alpha_{0,3}=\pm\alpha_{3,0}$ and $\alpha_{4,1}=\pm\alpha_{1,4}$, or $\alpha_{3,3}=\pm\alpha_{0,3}$ and $\alpha_{1,1}=\pm\alpha_{4,1}$. Moreover, in all scenarios we can reduce to $\tau=[X:Z:Y]$ via a change of variables $\phi$ in the normalizer of $\langle\sigma\rangle$, more precisely, via $\phi=\operatorname{diag}(1,\lambda,s\lambda)$ modulo $\langle[X:Z:Y],[Y:Z:X]\rangle$ with $\lambda^6=1$. That is, $\mathcal{C}_1$ is $K$-isomorphic to
\begin{eqnarray*}
  \mathcal{C}'_1&:&X^6+Y^6+Z^6+\alpha'_{4,1}X^4YZ+\alpha'_{3,3}X^3(Y^3+Z^3)+\alpha'_{2,2}X^2Y^2Z^2\\
  &+& \alpha'_{1,2}XYZ(Y^3+Z^3)+\alpha'_{0,3}Y^3Z^3=0.
\end{eqnarray*}
Here $\Aut(\mathcal{C}'_1)=\langle\sigma,\,\tau\rangle=\varrho_1\left(\operatorname{S}_3\right)$. In particular, we should impose $\alpha'_{4,1}\neq\alpha'_{1,2}$ or $\alpha'_{3,3}\neq\alpha'_{0,3}$ to avoid having $[Y:Z:X]$ as an extra automorphism. Also, $\left(\alpha'_{3,3},\alpha'_{1,2}\right)\neq(0,0)$ to avoid having $\operatorname{diag}(1,\zeta_6,\zeta_6^{-1})$ as an extra automorphism of order $6>3$.

This shows part of Theorem \ref{thm3}, (1)-(ii).
\item If $\Lambda(\Aut(C))=\operatorname{A}_4$, then the \href{https://people.maths.bris.ac.uk/~matyd/GroupNames/index500.html}{Group Structure of $\operatorname{A}_4$} assures that $\Lambda(\Aut(C))$ contains $\Lambda(\tau)$ and $\Lambda(\tau')$
both of order $2$ such that
$$\Lambda(\sigma)\Lambda(\tau)\Lambda(\sigma)^{-1}=\Lambda(\tau'),\,\Lambda(\sigma)\Lambda(\tau')\Lambda(\sigma)^{-1}=\Lambda(\tau')\Lambda(\tau)=\Lambda(\tau)\Lambda(\tau').$$
We aim to show that such $\tau$ and $\tau'$ do not exist.
Write $\Lambda(\tau)=\left(
                \begin{array}{cc}
                  a & b \\
                  c & d \\
                \end{array}
              \right)$, then being of order $2$ yields $(a+d)b=(a+d)c=0$ and $a=\pm d$. So
$\Lambda(\tau)=\left(
                \begin{array}{cc}
                  0 & b \\
                  c & 0 \\
                \end{array}
              \right)\,\,\,\text{or}\,\,\,\left(
                                              \begin{array}{cc}
                                                a & b \\
                                                c & -a \\
                                              \end{array}
                                            \right).$

- If $\Lambda(\tau)=\left(
                \begin{array}{cc}
                  0 & b \\
                  c & 0 \\
                \end{array}
              \right)$, then $$\Lambda(\tau')=\Lambda(\sigma)\Lambda(\tau)\Lambda(\sigma)^{-1}=\left(
                                                                                                    \begin{array}{cc}
                                                                                                      0 & \zeta_3^{-1}b \\
                                                                                                      \zeta_3c & 0 \\
                                                                                                    \end{array}
                                                                                                  \right)=\Lambda(\tau)\,\,\text{in}\,\,\PGL_2(K),$$which implies that $
                                              \Lambda(\tau')\Lambda(\tau)\neq\Lambda(\tau)\Lambda(\tau')    $
                               a contradiction.

- If $\Lambda(\tau)=\left(
                                              \begin{array}{cc}
                                                a & b \\
                                                c & -a \\
                                              \end{array}
                                            \right),$ then $\Lambda(\tau')=\Lambda(\sigma)\Lambda(\tau)\Lambda(\sigma)^{-1}=\left(
                                              \begin{array}{cc}
                                                a & \zeta_3^{-1}b \\
                                                \zeta_3c & -a \\
                                              \end{array}
                                            \right)$ such that $\Lambda(\tau)\Lambda(\tau')=\Lambda(\tau')\Lambda(\tau)$. That is,
             $$
             \left(
                                              \begin{array}{cc}
                                                a^2+\zeta_3bc & (\zeta_3^{-1}-1)ab \\
                                               (1- \zeta_3)ac & a^2+\zeta_3^{-1}bc \\
                                              \end{array}
                                            \right)=\left(
                                              \begin{array}{cc}
                                                a^2+\zeta_3^{-1}bc & -(\zeta_3^{-1}-1)ab \\
                                               -(1- \zeta_3)ac & a^2+\zeta_3bc \\
                                              \end{array}
                                            \right)\,\,\text{in}\,\,\PGL_2(K).$$

For this to be true, either $ab=ac=0$ or $a^2+\zeta_3bc=-(a^2+\zeta_3^{-1}bc).$
Assuming $ab=ac=0$ yields $\Lambda(\tau')=\left(
                                                                                                    \begin{array}{cc}
                                                                                                      0 & \zeta_3^{-1}b \\
                                                                                                      \zeta_3^{-1}c & 0 \\
                                                                                                    \end{array}
                                                                                                  \right)=\Lambda(\tau)\,\,\text{in}\,\,\PGL_2(K)$ or $\Lambda(\tau')=\left(
                                                                                                    \begin{array}{cc}
                                                                                                      a & 0 \\
                                                                                                      0 & -a \\
                                                                                                    \end{array}
                                                                                                  \right)=\Lambda(\tau)\,\,\text{in}\,\,\PGL_2(K)$, which is again a contradiction.
Assuming $a^2+\zeta_3bc=-(a^2+\zeta_3^{-1}bc)$ yields $c=2a^2/b$ with $ab\neq0$. 
Moreover, $\Lambda(\sigma)\Lambda(\tau')\Lambda(\sigma)^{-1}=\Lambda(\tau)\Lambda(\tau'),$ hence
$$
\left(
  \begin{array}{cc}
    a & \zeta_3b \\
    2a^2/b & -a \\
  \end{array}
\right)=
\left(
                                              \begin{array}{cc}
                                                a(\zeta_3-\zeta_3^{-1}) & (\zeta_3^{-1}-1)b \\
                                               2a^2(1- \zeta_3)/b & -a(\zeta_3-\zeta_3^{-1}) \\
                                              \end{array}
                                            \right)
                                            \,\,\text{in}\,\,\PGL_2(K).$$
This is valid only if $(\zeta_3-\zeta_3^{-1})\zeta_3=(\zeta_3^{-1}-1)$ and $(\zeta_3-\zeta_3^{-1})=(1-\zeta_3)$, however, the second equation is never valid. This means that $\Lambda(\Aut(C))\neq\operatorname{A}_4$.
\end{enumerate}
\item Thirdly, assume that $\mathcal{C}_i$ is a descendant of the Klein sextic curve $\mathcal{K}_6$.

\begin{claim}\label{claim12,312}
For $\mathcal{C}_1$ a descendant of $\mathcal{K}_6$, $\Aut(\mathcal{C}_1)=\varrho_2(\Z/3\Z)$.
\end{claim}
\begin{proof} (of Claim \ref{claim12,312})
If $\mathcal{C}_1$ is a descendant of $\mathcal{K}_6$ with bigger automorphism group than $\langle\sigma\rangle$, then, from the \href{https://people.maths.bris.ac.uk/~matyd/GroupNames/1/C7sC3.html}{Group Structure of $\Z/21\Z\rtimes\Z/3\Z$} and since the automorphisms of $C$ have orders $\leq3$, $\Aut(\mathcal{C}_1)$ should be conjugate to a $\left(\Z/3\Z\right)^2$ in $\Aut(\mathcal{K}_6)$. Thus $\mathcal{C}_1$ has another automorphism $\sigma'\notin\langle\sigma\rangle$ of order $3$ that commutes with $\sigma$. Direct calculations show that we can take 
$\sigma'=\operatorname{diag}(1,s,t)$ with $s^3=t^3=1$ or $[sY:tZ:X]$ with $s,t\in K^*$.

In the first case, $\sigma'$ reduces to an homology as $\sigma'\notin\langle\sigma\rangle$. This is absurd because $\Aut(\mathcal{K}_6)$ does not contain any homologies of period $3$.
Regarding the second case, any descendant $\mathcal{C}'$ of the Klein curve $\mathcal{C}':X^5Y+Y^5Z+Z^5X+$\,lower terms in $X,Y,Z$ satisfies the property that its automorphism group fixes the triangle $\Delta$ whose vertices are the three reference points $(1:0:0),\,(0:1:0)$ and $(0:0:1)$, moreover, those points all lie on $\mathcal{C}'$. Because $\Delta$ is the only triangle fixed by $\langle\sigma,[sY:tZ:X]\rangle$ for any $s,t$ and because none of its vertices lies on $\mathcal{C}_1$, we conclude that $\Aut(\mathcal{C}_1)$ can not equal $\langle\sigma,[sY:tZ:X]\rangle$. This proves the claim for $\mathcal{C}_1$.
\end{proof}

\begin{claim}\label{claim12,312v2}
For $\mathcal{C}_2$ a descendant of $\mathcal{K}_6$, $\Aut(\mathcal{C}_2)$ is either conjugate to $\varrho_2(\Z/3\Z)$ or $\varrho_2\left((\Z/3\Z)^2\right)$.
\end{claim}
\begin{proof} (of Claim \ref{claim12,312v2})
If $\mathcal{C}_2$ is a descendant of $\mathcal{K}_6$ with bigger automorphism group than $\langle\sigma\rangle$, then $\Aut(\mathcal{C}_2)=\langle\sigma,[sY:tZ:X]\rangle$ for some $s,t\in K^*$. For $\sigma'\in\Aut(\mathcal{C}_2)$,
$s=\zeta_{21}^r,\,$ $t=\zeta_{21}^{-4r},\,$ $\alpha_{0,2}=\zeta_{21}^{-12r}\alpha_{4,0},\,$ $\alpha_{2,4}=\zeta_{21}^{3r}\alpha_{4,0},\,$ $\alpha_{1,3}=\zeta_{21}^{-6r}\alpha_{3,2},\,$ $\alpha_{2,1}=\zeta_{21}^{3r}\alpha_{3,2},$ and $\mathcal{C}_2$ reduces to
\begin{eqnarray*}
  X^5Y&+&Y^5Z+XZ^5+\alpha_{4,0}\left(X^4Z^2+\zeta_{21}^{3r}X^2Y^4+\zeta_{21}^{-12r}Y^2Z^4\right)\\
  &+& \alpha_{3,2}XYZ\left(X^2Y+\zeta_{21}^{3r}XZ^2+\zeta_{21}^{-6r}Y^2Z\right)=0.
\end{eqnarray*}
In any situation, there exists a change of variables $\phi=\operatorname{diag}(1,\zeta_{21}^{r'},\zeta_{21}^{17r'})\in\Aut(\mathcal{K}_6)$ such that $21\,|\,18r'+r,\,12r'-4r$ for some $r'\in\{0,1,...,20\}$ that transforms $\mathcal{C}_2$ up to $K$-isomorphism to
\begin{eqnarray*}
  \mathcal{C}_2':X^5Y&+&Y^5Z+XZ^5+\alpha_{4,0}\zeta_{21}^{4r}\left(X^4Z^2+X^2Y^4+Y^2Z^4\right)\\
  &+& \alpha_{3,2}\zeta_{21}^{-r}XYZ\left(X^2Y+XZ^2+Y^2Z\right)=0,
\end{eqnarray*}
where $\Aut(\mathcal{C}_2')=\varrho_{2}\left((\Z/3\Z)^2\right)=\langle\sigma,[Y:Z:X]\rangle$. In particular, we must have $\left(\alpha_{2,4},\alpha_{1,3}\right)\neq(0,0)$ or $\operatorname{diag}(1,\zeta_{21},\zeta_{21}^{-4})\in\Aut(\mathcal{C}_2')$ of order $21>3$.

This completes the proof, which in turns shows Theorem \ref{thm3}, (2)-(i).
\end{proof}
\item Now, assume that $\mathcal{C}_i$ is a descendant of the Fermat curve $\mathcal{F}_6$. From the \href{https://people.maths.bris.ac.uk/~matyd/GroupNames/193/C6%5E2sS3.html}{Group structure of $\Aut(\mathcal{F}_6)$},
one sees that if $\mathcal{C}_i$ is a descendant of $\mathcal{F}_6$ with bigger automorphism group than $\langle\sigma\rangle$, then $\Aut(\mathcal{C}_i)$ is conjugate to one of the following groups inside $\Aut(\mathcal{F}_6)$: $$\left(\Z/3\Z\right)^2,\,\operatorname{S}_3,\,\operatorname{A}_4,\,\Z/3\Z\rtimes\operatorname{S}_3,\,\operatorname{He}_3.$$
In what follows, we treat each of these cases for $\mathcal{C}_1$ and $\mathcal{C}_2$ respectively, more precisely, Claim \ref{claim13,312v2} and Claim \ref{claim13,312v1} below.

\begin{claim}\label{claim13,312v2}
For $\mathcal{C}_1$ a descendant of $\mathcal{F}_6$, $\Aut(\mathcal{C}_1)$ is conjugate to $\varrho_2(\Z/3\Z),\,$ $\varrho_2\left(\operatorname{S}_3\right),\,$ $\varrho_1\left(\Z/3\Z\rtimes\operatorname{S}_3\right),\,$ $\varrho_{1}\left((\Z/3\Z)^2\right),\,\varrho_{2}\left((\Z/3\Z)^2\right)$ or $\varrho_1(\operatorname{A}_{4})$.
\end{claim}

\begin{claim}\label{claim13,312v1}
For $\mathcal{C}_2$ a descendant of $\mathcal{F}_6$, $\Aut(\mathcal{C}_2)$ is conjugate to $\varrho_2(\Z/3\Z),\,$ $\varrho_{2}((\Z/3\Z)^2)$ or $\varrho_2(\operatorname{A}_4)$.
\end{claim}
\end{itemize}

\begin{proof} (of Claim \ref{claim13,312v2})
- If $\Aut(\mathcal{C}_1)$ is conjugate to $\operatorname{S}_3$ or $\Z/3\Z\rtimes\operatorname{S}_3$ inside $\Aut(\mathcal{F}_6)$, then $\mathcal{C}_1$ has an involution $\tau$ such that $\tau\sigma\tau=\sigma^{-1}$. Similarly as before, this holds only if
      $\alpha_{3,3}=\pm\alpha_{3,0}$ and $\alpha_{1,1}=\pm\alpha_{1,4}$,\,$\alpha_{0,3}=\pm\alpha_{3,0}$ and $\alpha_{4,1}=\pm\alpha_{1,4}$, or $\alpha_{3,3}=\pm\alpha_{0,3}$ and $\alpha_{1,1}=\pm\alpha_{4,1}$. In this scenario, $\mathcal{C}_1$ is $K$-isomorphic to
\begin{eqnarray*}
  \mathcal{C}'_1&:&X^6+Y^6+Z^6+\alpha'_{4,1}X^4YZ+\alpha'_{3,3}X^3(Y^3+Z^3)+\alpha'_{2,2}X^2Y^2Z^2\\
  &+& \alpha'_{1,2}XYZ(Y^3+Z^3)+\alpha'_{0,3}Y^3Z^3=0,
\end{eqnarray*}
where $\varrho_2(\operatorname{S}_3)$ generated by $\sigma=\operatorname{diag}(1,\zeta_3,\zeta_3^{-1})$ and $\tau=[X:Z:Y]$ is a subgroup of $\Aut(\mathcal{C}'_1)$. Furthermore, if $\Aut(\mathcal{C}'_1)$ equals $\Z/3\Z\rtimes\operatorname{S}_3$, then it must contain another automorphism $\sigma'\notin\langle\sigma,\tau\rangle$ of order $3$ that commutes with $\sigma$ and satisfies $\tau\sigma'\tau=\sigma'^{-1}$.
Thus $\sigma'=[s'Y:s'^{-1}Z:X]$ and the invariance of the defining equation for $\mathcal{C}'_1$ under the action of $\sigma'$ yields $s'^3=1$,\,$\alpha'_{4,1}=\alpha'_{1,2}$ and $\alpha'_{3,3}=\alpha'_{0,3}$. Hence $\mathcal{C}'_1$ becomes
\begin{eqnarray*}
  X^6&+&Y^6+Z^6+\alpha'_{1,2}XYZ(X^3+Y^3+Z^3)+\alpha'_{3,3}(X^3Y^3+Y^3Z^3+Z^3X^3)\\
  &+&\alpha'_{2,2}X^2Y^2Z^2=0
\end{eqnarray*}
with $\Aut(\mathcal{C}'_1)=\varrho_1(\Z/3\Z\rtimes\operatorname{S}_3)$. This shows the rest of Theorem \ref{thm3}, (1)-(ii).

- If $\Aut(\mathcal{C}_1)$ is conjugate to $(\Z/3\Z)^2$ or $\operatorname{He}_3$ inside $\Aut(\mathcal{F}_6)$, then $\mathcal{C}_1$ would have an automorphism $\sigma'\notin\langle\sigma\rangle$ of order $3$ that commutes with $\sigma$ since every copy of $\Z/3\Z$ in any of these groups is contained in a $(\Z/3\Z)^2$. Similarly as before, we can take $\sigma'=\operatorname{diag}(1,s',t')$ with $s'^3=t'^3=1$ or $[s'Y:t'Z:X]$ with $s',t'\in K^*$.
\begin{enumerate}
  \item Suppose that $\sigma'=\operatorname{diag}(1,s',t')\in\Aut(\mathcal{C}_1)$. Because $\sigma'\notin\langle\sigma\rangle$, we have $\sigma'=\operatorname{diag}(1,1,\zeta_3),\,$ $\operatorname{diag}(\zeta_3,1,1)$ or $\operatorname{diag}(1,\zeta_3,1)$. 
      Consequently, $\alpha_{4,1}=\alpha_{2,2}=\alpha_{1,1}=\alpha_{1,4}=0$ and $\mathcal{C}_1$ reduces to
$$
  X^6+Y^6+Z^6+\alpha_{3,3}X^3Y^3+\alpha_{3,0}X^3Z^3+\alpha_{0,3}Y^3Z^3=0,
$$
with $\varrho_1((\Z/3\Z)^2)\subseteq\Aut(\mathcal{C}_1)$. On the other hand, $\Aut(\mathcal{C}_1)$ equals $\operatorname{He}_3$ only if it contains an extra automorphism $\sigma''\notin\langle\sigma,\sigma'\rangle$ of order $3$ that commutes with $\sigma$ and satisfies $\sigma''\sigma'\sigma''^{-1}=\sigma'\sigma^{-1}$. This gives us $\sigma''=[s''Y:t''Z:X]$ for some $s'',t''\in K^*$. Hence $s''^6=t''^6=1$,\, $\alpha_{3,3}=s''^3\alpha_{3,0},\,$ $\alpha_{0,3}=t''^3\alpha_{3,0}$, and $\mathcal{C}_1$ becomes of the form:  
$$X^6+Y^6+Z^6+\alpha_{3,0}\left(\pm X^3Y^3+X^3Z^3+t''^3Y^3Z^3\right)=0.$$ 
In particular, $[Y:X:t''Z]$ is an automorphism for $\mathcal{C}_1$ of order divisible by $2$. This is a contradiction as $2\nmid|\operatorname{He}_3|(=27)$.
  \item Suppose that $\sigma'_{s,t}=[s'Y:t'Z:X]\in\Aut(\mathcal{C}_1)$.
For this to be true, we should have $s'^6=t'^6=1,\,$ $\alpha_{4,1}=s't'\alpha_{1,1},\,$ $\alpha_{1,4}=s'^5t'^2\alpha_{1,1}$,\, $\alpha_{3,3}=s'^3\alpha_{3,0},\,$ $\alpha_{0,3}=t'^3\alpha_{3,0}=\pm\alpha_{3,0}$, and $\mathcal{C}_1$ is defined by
\begin{eqnarray*}
X^6&+&Y^6+Z^6+\alpha_{1,1}XYZ(s't'X^3\pm\frac{1}{s't'}Y^3+ Z^3)+\alpha_{2,2}X^2Y^2Z^2\\
  &+&\alpha_{3,0}(s'^3 X^3Y^3+X^3Z^3\pm Y^3Z^3)=0,
\end{eqnarray*}
such that $(s't')^2=1$ whenever $\alpha_{2,2}\neq0$. Consequently, it must be the case that $\alpha_{2,2}=0$ and $\alpha_{1,1}\neq0$ or $[t'Y:t'^{-1}X:Z]$ would be an extra involution, which violates the fact that $|\operatorname{Aut}(\mathcal{C}_1)|=9$ or $27$. That is, $s't'=\zeta_6^{\ell}$ for some $\ell\neq 0$ or $3\,\operatorname{mod}\,6$, and $\mathcal{C}_1$ becomes
\begin{eqnarray*}
 X^6+Y^6+Z^6&+&\alpha_{1,1}XYZ(\zeta_6^{\ell}X^3\pm\zeta_6^{-\ell}Y^3+ Z^3)+\\
 &+&\alpha_{3,0}( \pm (-1)^{\ell} X^3Y^3+X^3Z^3\pm Y^3Z^3)=0,
   \end{eqnarray*}
   for some $\alpha_{1,1}, \alpha_{3,0}\in K^*$.  
   Applying the projective change of variables $\phi=\operatorname{diag}(1,\frac{\sqrt[3]{s't'}}{s'},\frac{1}{\sqrt[3]{s't'}})$ we get
\begin{eqnarray*}
\mathcal{C}''_1:X^6+\zeta_{6}^{2\ell}Y^6+\zeta_6^{-2\ell}Z^6&+&\alpha'_{1,1}XYZ(X^3+\zeta_6^{2\ell}Y^3+ \zeta_6^{-2\ell}Z^3)+\\
 &+&\alpha'_{3,0}(  X^3Y^3+\zeta_6^{-2\ell}X^3Z^3+ \zeta_6^{2\ell}Y^3Z^3)=0.
   \end{eqnarray*}
Now with $\sigma$ and $\sigma'=[Y:Z:X]$ as automorphisms for  $\mathcal{C}''_1$, we have that  $\langle\sigma,\sigma'\rangle=\varrho_2((\Z/3\Z)^2)\subseteq\Aut(\mathcal{C}''_1)$. Again it is impossible that we can enlarge $\operatorname{Aut}(\mathcal{C}''_1)$ to $\operatorname{He}_3$, since this requires $\operatorname{diag}(1,1,\zeta_3)$ to be in $\operatorname{Aut}(\mathcal{C}''_1)$. This cannot be as $\alpha'_{1,1}=\frac{\alpha_{1,1}\zeta_6^{\ell}}{s}\neq0$.

\end{enumerate}

- If $\Aut(\mathcal{C}_1)$ is conjugate to an $\operatorname{A}_4$ inside $\Aut(\mathcal{F}_6)$, then it should be $\varrho_i(\operatorname{A}_{4})$ with $i=1$ or $2$.
\begin{enumerate}
  \item First, suppose that $\phi^{-1}\Aut(\mathcal{C}_1)\phi=\varrho_1(\operatorname{A}_{4})$. As all subgroups of $\operatorname{A}_{4}$ of order $3$ are $\operatorname{A}_{4}$-conjugated, there is no loss of generality to take
$\phi^{-1}\sigma\phi=[Y:Z:X]$ or $[Z:X:Y]$ . In particular, $\phi$ has one of the following shapes:
$$\phi_1:=
\left(
  \begin{array}{ccc}
    1 & 1 & 1 \\
    \lambda & \zeta_{3}^{-1}\lambda & \zeta_3\lambda \\
   \mu & \zeta_{3}\mu & \zeta_3^{-1}\mu \\
  \end{array}
\right),\,\phi_2:=\left(
  \begin{array}{ccc}
    \mu & \zeta_{3}\mu & \zeta_3^{-1}\mu \\
    1 & 1 & 1 \\
       \lambda & \zeta_{3}^{-1}\lambda & \zeta_3\lambda \\
  \end{array}\right),\,
  \phi_3:=\left(
  \begin{array}{ccc}
    \lambda & \zeta_{3}^{-1}\lambda & \zeta_3\lambda \\
   \mu & \zeta_{3}\mu & \zeta_3^{-1}\mu \\
    1 & 1 & 1 \\
  \end{array}
\right),$$
$$\phi_4:=
\left(
  \begin{array}{ccc}
    1 & 1 & 1 \\
    \lambda & \zeta_{3}\lambda & \zeta_3^{-1}\lambda \\
   \mu & \zeta_{3}^{-1}\mu & \zeta_3\mu \\
  \end{array}
\right),\,\phi_5:=\left(
  \begin{array}{ccc}
    \mu & \zeta_{3}^{-1}\mu & \zeta_3\mu \\
    1 & 1 & 1 \\
       \lambda & \zeta_{3}\lambda & \zeta_3^{-1}\lambda \\
  \end{array}\right),\,
  \phi_6:=\left(
  \begin{array}{ccc}
    \lambda & \zeta_{3}\lambda & \zeta_3^{-1}\lambda \\
   \mu & \zeta_{3}^{-1}\mu & \zeta_3\mu \\
    1 & 1 & 1 \\
  \end{array}
\right),
$$
for some $\lambda,\mu\in K^*$.

Now, we handle each of these situations to determine the restrictions on the defining equation of $\mathcal{C}_1$ for which this holds.
\begin{itemize}
  \item For $\phi_{1}\operatorname{diag}(1,1,-1)\phi_{1}^{-1}$ (respectively $\phi_{4}\operatorname{diag}(1,1,-1)\phi_{4}^{-1}$) to be in $\Aut(\mathcal{C}_1)$, we must eliminate the coefficients of $X^5Z,\,X^5Y,\,Y^5Z,\,XZ^5,\,$ $YZ^5,\,$ $X^4Y^2,\,X^4Z^2$ from the transformed equation $^{\phi_{i}\operatorname{diag}(1,1,-1)\phi_{i}^{-1}}\mathcal{C}_1=\mathcal{C}_1$ with $i=1$ and $4$ respectively. In this way, we obtain:
\begin{eqnarray*}
  \alpha_{4,1}&=&\dfrac{2\left(29-54\lambda^6-54\mu^6\right)}{27\lambda\mu},\,\alpha_{3,3}=\dfrac{2\left(81\mu^6-27\lambda^6-26\right)}{27\lambda^3},\\
  \alpha_{3,0}&=&\dfrac{2\left(81\lambda ^6-27\mu^6-26\right)}{27\mu ^3},\,\alpha_{1,4}=\dfrac{2\left(27\lambda^6-54\mu ^6-52\right)}{27\lambda^4\mu},\\
\alpha_{1,1}&=&\dfrac{2\left(27\mu^6-54\lambda^6-52\right)}{27\lambda\mu ^4},\,\alpha_{0,3}=\dfrac{2\left(82-27\lambda^6-27\mu ^6\right)}{27\lambda^3\mu^3},\\
\alpha_{2,2}&=&\dfrac{9\lambda^6+9\mu^6+10}{3\lambda^2\mu^2}.
\end{eqnarray*}
In particular, $\mathcal{C}_1$ is $K$-isomorphic via $\phi_1$  (respectively $\phi_4$ followed by $Y\leftrightarrow Z$) to
$\mathcal{C}_{1,\lambda,\mu}$ described in Theorem \ref{thm3}, (1)-(iii).
\item For $\phi_{2}\operatorname{diag}(1,1,-1)\phi_{2}^{-1}$ (respectively $\phi_{5}\operatorname{diag}(1,1,-1)\phi_{5}^{-1}$) to be in $\Aut(\mathcal{C}_1)$, one notices that $\phi_2=[Z:X:Y]\,\phi_1=\phi_1\,\circ\,[Z:X:Y]$ (respectively $\phi_5=[Z:X:Y]\,\phi_4=\phi_4\,\circ\,[Z:X:Y]$). This means that we get the same conclusion as above up to a permutation of the parameters, more precisely, after
      \begin{eqnarray*}
      \left(\alpha_{4,1},\alpha_{1,1},\alpha_{1,4}\right)&\mapsto&\left(\alpha_{1,1},\alpha_{1,4},\alpha_{4,1}\right),\\
      \left(\alpha_{0,3},\alpha_{3,3},\alpha_{3,0}\right)&\mapsto&\left(\alpha_{3,3},\alpha_{3,0},\alpha_{0,3}\right).
      \end{eqnarray*}
In other words, we have $\phi_{i}\operatorname{diag}(1,1,-1)\phi_{i}^{-1}$ with $i=2$ or $5$ inside $\Aut(\mathcal{C}_1)$ only if
\begin{eqnarray*}
  \alpha_{1,4}&=&\dfrac{2\left(29-54\lambda^6-54\mu^6\right)}{27\lambda\mu},\,\alpha_{0,3}=\dfrac{2\left(81\mu^6-27\lambda^6-26\right)}{27\lambda^3},\\
  \alpha_{3,3}&=&\dfrac{2\left(81\lambda ^6-27\mu^6-26\right)}{27\mu ^3},\,\alpha_{1,1}=\dfrac{2\left(27\lambda^6-54\mu ^6-52\right)}{27\lambda^4\mu},\\
\alpha_{4,1}&=&\dfrac{2\left(27\mu^6-54\lambda^6-52\right)}{27\lambda\mu ^4},\,\alpha_{3,0}=\dfrac{2\left(82-27\lambda^6-27\mu ^6\right)}{27\lambda^3\mu^3},\\
\alpha_{2,2}&=&\dfrac{9\lambda^6+9\mu^6+10}{3\lambda^2\mu^2}.
\end{eqnarray*}
Once more $\mathcal{C}_1$ reduces to $\mathcal{C}_{1,\lambda,\mu}$ described in Theorem \ref{thm3}, (1)-(iii).

Similarly, $\phi_3=\phi_1\,\circ\,[Y:Z:X]$ and $\phi_6=\phi_4\,\circ\,[Y:Z:X]$. So $\phi_{i}\operatorname{diag}(1,1,-1)\phi_{i}^{-1}$ with $i=3$ or $6$ is an automorphism for $\mathcal{C}_1$ only if
\begin{eqnarray*}
  \alpha_{1,1}&=&\dfrac{2\left(29-54\lambda^6-54\mu^6\right)}{27\lambda\mu},\,\alpha_{3,0}=\dfrac{2\left(81\mu^6-27\lambda^6-26\right)}{27\lambda^3},\\
  \alpha_{0,3}&=&\dfrac{2\left(81\lambda ^6-27\mu^6-26\right)}{27\mu ^3},\,\alpha_{4,1}=\dfrac{2\left(27\lambda^6-54\mu ^6-52\right)}{27\lambda^4\mu},\\
\alpha_{1,4}&=&\dfrac{2\left(27\mu^6-54\lambda^6-52\right)}{27\lambda\mu ^4},\,\alpha_{3,3}=\dfrac{2\left(82-27\lambda^6-27\mu ^6\right)}{27\lambda^3\mu^3},\\
\alpha_{2,2}&=&\dfrac{9\lambda^6+9\mu^6+10}{3\lambda^2\mu^2},
\end{eqnarray*}
where $\mathcal{C}_1$ becomes $K$-isomorphism to $C_{1,\lambda,\mu}$.

This shows Theorem \ref{thm3}, (1)-(iii).

\end{itemize}
  \item Second, suppose that $\psi^{-1}\Aut(\mathcal{C}_1)\psi=\varrho_2(\operatorname{A}_{4})$. Again, we can impose $\psi^{-1}\sigma\psi=[\zeta_6^{-1}Y:Z:X]$ or $[Z:\zeta_6X:Y]$, in particular, $\psi$ has the shape of $\psi_i$ below.
$$\psi_1:=
\left(
  \begin{array}{ccc}
    1 & \zeta_{18}^{-2} & \zeta_{18}^{-1} \\
    \lambda & \zeta_{18}^{-8}\lambda & \zeta_{18}^{5}\lambda \\
   \mu & \zeta_{18}^{4}\mu & \zeta_{18}^{-7}\mu \\
  \end{array}
\right),\,\psi_2:=\left(
  \begin{array}{ccc}
\mu & \zeta_{18}^{4}\mu & \zeta_{18}^{-7}\mu \\
    1 & \zeta_{18}^{-2} & \zeta_{18}^{-1}\\
   \lambda & \zeta_{18}^{-8}\lambda & \zeta_{18}^{5}\lambda \\
  \end{array}\right),\,
  \psi_3:=\left(
  \begin{array}{ccc}
   \lambda & \zeta_{18}^{-8}\lambda & \zeta_{18}^{5}\lambda \\
\mu & \zeta_{18}^{4}\mu & \zeta_{18}^{-7}\mu \\
   1 & \zeta_{18}^{-2} & \zeta_{18}^{-1}\\
  \end{array}
\right),$$
$$\psi_4:=
\left(
  \begin{array}{ccc}
    1 & \zeta_{18}^{2} & \zeta_{18} \\
    \lambda & \zeta_{18}^{-4}\lambda & \zeta_{18}^{7}\lambda \\
   \mu & \zeta_{18}^{8}\mu & \zeta_{18}^{-5}\mu \\
  \end{array}
\right),\,\psi_5:=\left(
  \begin{array}{ccc}
   \mu & \zeta_{18}^{8}\mu & \zeta_{18}^{-5}\mu \\
       1 & \zeta_{18}^{2} & \zeta_{18} \\
    \lambda & \zeta_{18}^{-4}\lambda & \zeta_{18}^{7}\lambda \\
  \end{array}\right),\,
  \psi_6:=\left(
  \begin{array}{ccc}
    \lambda & \zeta_{18}^{-4}\lambda & \zeta_{18}^{7}\lambda \\
\mu & \zeta_{18}^{8}\mu & \zeta_{18}^{-5}\mu \\
   1 & \zeta_{18}^{2} & \zeta_{18}\\
  \end{array}
\right),$$
for some $\lambda,\mu\in K^*$. However, it is straightforward to check that none of these transformation transforms $\mathcal{C}_1$ to $\mathcal{C}'$ whose core is $X^6+Y^6+Z^6$. Consequently, $\mathcal{C}_1$ is never a descendant of the Fermat curve $\mathcal{F}_6$ with $\Aut(\mathcal{C}_1)$ conjugate to $\varrho_2(\operatorname{A}_{4})$. 
\end{enumerate}
This proves Claim \ref{claim13,312v2}.
\end{proof}

It remains to prove Claim \ref{claim13,312v1} for $\mathcal{C}_2$ that is a descendant of the Fermat curve $\mathcal{F}_6$.
\begin{proof} (of Claim \ref{claim13,312v1})
- We easily discard the cases when $\Aut(\mathcal{C}_2)$ equals an $\operatorname{S}_3$ or $\Z/3\Z\rtimes\operatorname{S}_3$ inside $\Aut(\mathcal{F}_6)$ as none of the involutions $[X:sZ:s^{-1}Y],\,[sY:s^{-1}X:Z]$ and $[sZ:Y:s^{-1}X]$ preserves the core $X^5Y+Y^5Z+Z^5X$ of $\mathcal{C}_2$.

- On the other hand, if $\Aut(\mathcal{C}_2)$ equals $(\Z/3\Z)^2$ or $\operatorname{He}_3$, then the discussion we had to show Claim \ref{claim12,312v2} applies to conclude that $\mathcal{C}_2$ is $K$-isomorphic to
\begin{eqnarray*}
  \mathcal{C}':X^5Y&+&Y^5Z+XZ^5+\alpha_{4,0}\zeta_{21}^{4r}\left(X^4Z^2+X^2Y^4+Y^2Z^4\right)\\
  &+& \alpha_{3,2}\zeta_{21}^{-r}XYZ\left(X^2Y+XZ^2+Y^2Z\right)=0,
\end{eqnarray*}
where $\varrho_2((\Z/3\Z)^2)\subseteq\Aut(\mathcal{C}')$. Next, if $\Aut(\mathcal{C}')$ is $\operatorname{He}_3$, then there must be another automorphism $\sigma'\notin\varrho_2((\Z/3\Z)^2)$ of order $3$ that commutes with $\sigma$ such that $\sigma'\,[Y:Z:X]\,\sigma'^{-1}=[Y:Z:X]\,\sigma^{-1}$. Straightforward calculations show that $\sigma'=[s'Y:t'Z:X]$ or $[s'Z:t'X:Y]$ with $s't'=\zeta_3$ and $s'^2t'^{-1}=\zeta_3^{-1}$. So $\sigma'$ belongs to $\varrho_1((\Z/3\Z)^2)$ modulo $\langle[Y:Z:X]\rangle$.
Obviously, none of these transformations leaves invariant the core of $\mathcal{C}'$. Therefore, $\Aut(\mathcal{C}_2)$ is never conjugate to $\operatorname{He}_3$ inside $\mathcal{F}_6$.

- Thirdly, following the notations of Claim \ref{claim13,312v2}, a change of variables of the form $\phi=\phi_i$ for $i=1,2,...,6$ does not transform $\mathcal{C}_2$ to $\mathcal{C}'_2:X^6+Y^6+Z^6+$ lower order terms in $X,Y,Z$. Thus $\mathcal{C}_2$ is not a descendant of $\mathcal{F}_6$ such that $\phi^{-1}\Aut(\mathcal{C}_2)\phi=\varrho_1(\operatorname{A}_{4})$.
On the other hand, $\psi_{i}\operatorname{diag}(1,1,-1)\psi_{i}^{-1}\in\Aut(\mathcal{C}_2)$ with $i=1$ or $4$ only if
\begin{eqnarray*}
  \alpha_{2,4}&=&\dfrac{\lambda^5\mu+4\mu^5}{2\lambda^4},\,\alpha_{4,0}=\dfrac{\lambda+4\lambda^5\mu}{2\mu^2},\,\alpha_{0,2}=\dfrac{4\lambda+\mu^5}{2\lambda^2\mu^4}\\
  \alpha_{1,3}&=&\dfrac{2\left(2\lambda^5\mu+2\lambda+\mu^5\right)}{\lambda^3\mu^2},\,\alpha_{3,2}=\dfrac{2\lambda^5\mu+4\lambda+4\mu^5}{\lambda^2\mu},\,\alpha_{2,1}=\frac{2\left(2\lambda^5\mu+\lambda+2\mu^5\right)}{\lambda\mu^3}.
\end{eqnarray*}
The above restrictions are consequences of eliminating the coefficients of $X^6,\,Y^6,\,Z^6,\,$ $X^5Z,\,Y^4Z^2,\,$ $X^4Y^2,\,$ $X^4Z^2$ from the transformed equation $^{\psi_{i}\operatorname{diag}(1,1,-1)\psi_{i}^{-1}}\mathcal{C}_2=\mathcal{C}_2$. Moreover, $\mathcal{C}_2$ is $K$-isomorphic via $\psi_1$  (respectively $\psi_4$ followed by $Y\leftrightarrow Z$) to $\mathcal{C}_{2,\lambda,\mu}$ described in Theorem \ref{thm3}, (2)-(ii).
The rest is obvious by noticing that $\psi_2=\psi_1\,\circ\,[Z:X:Y],\,$ $\psi_5=\phi_4\,\circ\,[Z:X:Y],\,$ $\psi_3=\psi_1\,\circ\,[Y:Z:X]$ and $\psi_6=\psi_4\,\circ\,[Y:Z:X]$.

This proves Claim \ref{claim13,312v1}.
\end{proof}

\end{document}